\documentclass[12pt]{amsart}

\usepackage{amsmath,amssymb,amsthm}
\usepackage{graphicx}
\usepackage{enumerate}
\usepackage{color}

\newtheorem{theorem}{Theorem}[section]

\newtheorem{corollary}[theorem]{Corollary}
\newtheorem{lemma}[theorem]{Lemma}

\theoremstyle{definition}

\newtheorem*{claim*}{Claim}

\newtheorem*{question*}{Question}
\newtheorem*{answer*}{Answer}
\newtheorem*{application*}{Application}

\theoremstyle{remark}
\newtheorem{remark}[theorem]{Remark}
\newtheorem*{remark*}{Remark}

\newcommand{\Sr}{$S\times \mathbb{R}$ }
\newcommand{\Hc}{$\mathcal{HC}(S, \alpha)$}

\newcommand{\param}{{\mathchoice{\mkern1mu\mbox{\raise2.2pt\hbox{$
\centerdot$}}
\mkern1mu}{\mkern1mu\mbox{\raise2.2pt\hbox{$\centerdot$}}\mkern1mu}{
\mkern1.5mu\centerdot\mkern1.5mu}{\mkern1.5mu\centerdot\mkern1.5mu}}}

\renewcommand{\setminus}{{\smallsetminus}}

\begin{document}

\title       {A Curve Complex and Incompressible Surfaces in $S\times \mathbb{R}$}
\author   {Ingrid Irmer}
\address {Department of Mathematics\\
               National University of Singapore\\
               Block S17
               10, Lower Kent Ridge Road
               Singapore 119076 }
\email      {matiim@nus.edu.sg}
\maketitle

\begin{abstract}
Various curve complexes with vertices representing multicurves on a surface $S$ have been defined, for example \cite{BBM}, \cite{Hatcher} and \cite{Me2}. The homology curve complex $\mathcal{HC}(S,\alpha)$ defined in \cite{Me} is one such complex, with vertices corresponding to multicurves in a nontrivial integral homology class $\alpha$. Given two multicurves $m_1$ and $m_2$ corresponding to vertices in $\mathcal{HC}(S,\alpha)$, it was shown in \cite{Me2} that a path in $\mathcal{HC}(S,\alpha)$ connecting these vertices represents a surface in $S\times \mathbb{R}$, and a simple algorithm for constructing minimal genus surfaces of this type was obtained. In this paper, a Morse theoretic argument will be used to prove that all embedded orientable incompressible surfaces in $S\times \mathbb{R}$ with boundary curves homotopic to $m_{2}-m_1$ are homotopic to a surface constructed in this way. This is used to relate distance between two vertices in $\mathcal{HC}(S,\alpha)$ to the Seifert genus of the corresponding link in $S\times \mathbb{R}$.

\end{abstract}

\section{Introduction} \label{Intro}
Suppose $S$ is a closed oriented surface, each connected component of which has genus at least 2. Let $\pi$ be the projection of $S\times \mathbb{R}$ onto $S\times 0$ given by $(s,r)\mapsto (s,0)$. A \textit{multicurve} in $S\times \mathbb{R}$ is a one dimensional embedded submanifold that projects onto a one dimensional embedded submanifold of $S\times 0$. It is also assumed that multicurves do not contain curves that bound discs.\\

Fix a nontrivial element $\alpha$ of $H_{1}(S,\mathbb{Z})$. The \textit{homology curve complex}, $\mathcal{HC}(S,\alpha)$, is a simplicial complex whose vertex set is the set of all isotopy classes of oriented multicurves in $S$ in the homology class $\alpha$. A set of vertices $m_{1},\ldots, m_k$ spans a simplex if the representatives of the isotopy classes can all be chosen to be disjoint. \\

As described in \cite{Me2}, a path in $\mathcal{HC}(S,\alpha)$ corresponds to a surface in $S\times \mathbb{R}$. This construction will be briefly repeated in section \ref{surfaceconstruction}. The main theorem of this paper, theorem \ref{minimalgenus}, shows a converse of this.\\

In the special case that $H$ is an embedded surface with boundary contained in two level sets, theorem \ref{minimalgenus} is known; for example section 2 of \cite{Floyd}. 

All homotopies of surfaces in $S\times \mathbb{R}$ are assumed to be smooth, and are allowed to move the boundaries of surfaces. Let $m_1$ and $m_2$ be homologous multicurves in $S\times \mathbb{R}$ representing vertices in $\mathcal{HC}(S,\alpha)$.

\begin{theorem}
\label{minimalgenus}

Suppose $H$ is an oriented, embedded, connected, incompressible surface in $S\times \mathbb{R}$ with boundary $m_{2}-m_1$. Then there exists a path $\gamma$ in $\mathcal{HC}(S,\alpha)$ connecting the vertices corresponding to $m_1$ and $m_2$ such that $H$ is homotopic to an embedded surface constructed from $\gamma$.
\end{theorem}



The basic idea of the proof is to use the coordinate obtained by projecting the surface onto the second component of \Sr\  as a Morse function $m_{R}$ on the surface. This is shown to be possible in lemma \ref{morsetheory}.\\

An edge in $\mathcal{HC}(S,\alpha)$ represents the boundary of a union of surfaces that project one to one onto subsurfaces of $S\times \{0\}$. Given a path $\gamma$ in $\mathcal{HC}(S,\alpha)$ connecting the vertices $m_1$ and $m_2$, a surface with boundary $m_{2}-m_{1}$ is constructed by gluing together subsurfaces represented by edges. This construction is discussed in more detail in the next section.\\



The geometric intersection number, $i(m_{1},m_{2})$, of two multicurves in $S\times \mathbb{R}$ is defined by projecting onto $S\times 0$. Recall that the geometric intersection number of two multicurves $m_1$ and $m_2$ is the minimum possible number of intersections between a pair of multicurves, one of which is isotopic to $m_1$ and the other to $m_2$.\\

The \textit{distance}, $d_{\mathcal{H}}(v_{1},v_{2})$, between two vertices $v_1$ and $v_2$ in $\mathcal{HC}(S,\alpha)$ is defined to be the distance in the path metric of the one-skeleton, where all edges have length one.\\


In \cite{Me2} it was shown that the smallest possible genus of a surface in \Sr\  with boundary homotopic to $m_{2}-m_1$ provides a bound from below on distance in $\mathcal{HC}(S,\alpha)$ between two vertices represented by multicurves $m_1$ and $m_2$. Theorem \ref{minimalgenus} shows the converse, namely, the smallest possible genus of a surface in \Sr\  with boundary curves homotopic to $m_{2}-m_1$ gives a bound from above on the distance in $\mathcal{HC}(S,\alpha)$ between $m_1$ and $m_2$. A corollary of theorem \ref{minimalgenus} is used in \cite{Me2} to obtain a simple, $\mathcal{O}(i(m_{1},m_{2}))$ algorithm for constructing minimal genus surfaces in \Sr\  with boundary $m_{2}-m_{1}$. This is in contrast to the problem of finding a minimal genus surface of a knot embedded in a general 3-manifold, which was shown in \cite{AHT} to be \textbf{NP}-complete. 

\begin{corollary}[Distance in \Hc\ and genus of surfaces]
Let $d_{\mathcal{C}}(m_{1},m_{2})$ be the distance in $\mathcal{HC}(S,\alpha)$ between the vertices corresponding to the multicurves $m_1$ and $m_2$, and $g_H$ be the smallest possible genus of a surface in \Sr\  with boundary $m_2-m_1$. Then there exist constants $k_1$ and $k_2$ depending only on the genus of $S$ such that
\begin{equation*}
k_{1}g_{H} \leq d_{\mathcal{C}}(m_{1},m_{2}) \leq k_{2}g_{H}
\end{equation*}
\end{corollary}
\subsection*{Acknowledgements} I would like to thank Ursula Hamenst\"adt for her supervision of this project, and Dan Margalit for his advice.

\section{Constructing surfaces from paths.}
\label{surfaceconstruction}
Whenever this does not lead to confusion, the same symbol will be used for a vertex in $\mathcal{HC}(S,\alpha)$ and a multicurve in the corresponding isotopy class on $S$. Also, a path in $\mathcal{HC}(S,\alpha)$ will often be denoted by a sequence of multicurves, $m_{1}, m_{2},\ldots, m_{n}$ with the property that $m_i$ and $m_{i+1}$ are disjoint for every $1\leq i \leq n$, i.e. $m_i$ and $m_{i+1}$ represent an edge in $\mathcal{HC}(S, \alpha)$.\\

Let $\gamma:=\left\{\gamma_{0}, \gamma_{1}, \ldots\gamma_{j}\right\}$ be the vertices of a path in $\mathcal{HC}(S,\alpha)$. Since $S$ maps into $S\times \mathbb{R}$, each $\gamma_i$ represents a multicurve in $S\times \mathbb{R}$.\\

Consider constructing a surface $T_{\gamma}$ contained in $S\times \left\{j\right\}\subset S\times \mathbb{R}$ inductively as follows. Given $\gamma_0$, suppose we can isotope $\gamma_1$ such that there is a subsurface $S_1$ of $S$ with boundary $\gamma_{1}-\gamma_0$. Let $T_1$ be the surface in $S\times [0,1]$ given by $\gamma_{0}\times [0,\frac{1}{2}]\cup S_{1}\times \{\frac{1}{2}\}\cup \gamma_{1}\times [\frac{1}{2},1]$. Next, suppose we can isotope $\gamma_2$ so that there is a subsurface $S_2$ of $S$ with $\partial S_{2}=\gamma_{2}-\gamma_1$ and let $T_{2}=\gamma_{1}\times [1,\frac{3}{2}]\cup S_{2}\times \{\frac{3}{2}\}\cup \gamma_{2}\times [\frac{3}{2},2]$. Repeat this successively for each of the $\gamma_i$ until an embedded surface $T_{\gamma}=T_{1}\cup T_{2}\cup\ldots\cup T_{j}$ in $S\times [0,j]$ is obtained. To ensure that this construction yields an embedded surface, it is assumed that any null homologous submulticurve of $\gamma_{i+1}-\gamma_i$ of the form $c-c$ bounds an annulus, not the empty set.\\

\textbf{Simple paths and embeddedness}. A path in $\mathcal{HC}(S,\alpha)$ for which a surface can be constructed as described in the previous paragraph is called a \textit{simple path}. It is not hard to show that the definition is symmetric in $\gamma_0$ and $\gamma_j$. Unlike in \cite{Me2}, we do not require the components of $S_i$ to be oriented as subsurfaces of $S$. As a result, it will be clear from the construction that all the paths constructed in the proof of theorem \ref{minimalgenus} are simple. It follows that the resulting surfaces are embedded.\\

When constructing a surface in \Sr\  from a path $\gamma$ in \Hc, at every step there may be a choice involved as to whether to attach $S_i$, or its complement in $S\times 0$ with the opposite orientation. Call all such surfaces \textit{surfaces constructed from} $\gamma$. Given a surface, a corresponding path in $\mathcal{HC}(S, \alpha)$ is not generally unique.

\section{Morse theory with boundary}

All multicurves and manifolds are assumed to be smooth throughout this paper. The manifold $S\times \mathbb{R}$ is given a product metric $ds^{2}_{M}=ds^{2}_{S}+dR^{2}$ where $ds_{S}$ is a choice of metric on $S\times 0$. Similarly, $H$ and all surfaces in \Sr\  homotopic to $H$ are assumed to be covered by coordinate charts $(U_{1}, s_{1}, R),\ldots,(U_{k}, s_{k}, R)$, where the $s_i$ are coordinates obtained by projecting onto $S\times 0$.  \\

Whenever $H$ is embedded in \Sr, it is a corollary of theorem \ref{minimalgenus} that there is a homotopy of $H$ that takes $H$ to an embedded surface in $S\times [a,b]$ with boundary contained in the level sets $S\times \left\{a\right\}$ and $S\times \left\{b\right\}$. In order to work with surfaces whose boundaries are not contained in level sets, the standard Morse theory has to be modified slightly.\\

A \textit{critical point} of a Morse function on $\partial H$ is a point where the restriction of the Morse function to the boundary has zero derivative, and a \textit{degenerate critical point} on the boundary is any critical point that is not an isolated local extremum. \\








\begin{lemma}
\label{morsetheory}
Suppose $H$ is a compact embedded surface in $M$ with boundary consisting of the multicurves $m_1$ and $m_2$. Then there is an embedded surface in $S\times \mathbb{R}$, call it $H^{'}$, with the following properties:
\begin{enumerate}
\item $H^{'}$ is homotopic to $H$
\item The restriction, $m_{R}$, of the $\mathbb{R}$ coordinate to $H^{'}$ is a Morse function
\item No two critical points of the Morse function from 2 have the same value of the $\mathbb{R}$ coordinate. 

\end{enumerate}

\end{lemma}
\begin{proof}
It is a standard result, e.g. \cite{Milnor} Theorem 2.7, that on a compact manifold without boundary, the Morse functions form an open, dense (in the $\mathcal{C}^2$ topology) subset of the set of all smooth functions of the manifold into $\mathbb{R}$. This and similar standard results in Morse theory are proven by altering a given function by adding arbitrarily small functions with small derivatives. Similar arguments are used here; the main difference is that the coordinate $R$ is treated as fixed while the subset of \Sr\  to which $R$ is restricted is altered by a homotopy. \\

In the proof that the Morse functions form an open dense subset of $H$ into $\mathbb{R}$, first of all the existence of a surface $H^1$ homotopic to $H$ on which $\mathbb{R}$ is a Morse function on some neighbourhood of the boundary will be shown. The standard Morse theory arguments (e.g. theorem 2.7 of \cite{Milnor}) that assume empty boundary then apply to $H^1$, from which claims 2 and 3 of the lemma follow. It will then be shown that if the homotopies representing these alterations are sufficiently close to the identity, embeddedness is preserved.\\


\textbf{Existence of a homotopy of $H$ that makes $m_{R}$ Morse on some neighbourhood of the boundary.} Let $N$ be a collar of the boundary of $H$; the existence of which is guaranteed by theorem 6.1, chapter 4 of \cite{Hirsch}. The boundary of $H$ is a compact manifold without boundary, so by theorem 2.7 of \cite{Milnor}, if the restriction of $R$ to $\partial H$ is not a Morse function, there is a Morse function $R_m$ on $\partial H$ arbitrarily close to $R$ in the $\mathcal{C}^2$ topology. \\

The collar $N$ is diffeomorphic to several copies of $S^{1}\times [0,\iota]$, which defines coordinates $(t,r)$ on each component of $N$, where $t$ is the parameter on $S^1$ and $r$ is defined on the interval $[o,\iota]$ and is equal to zero on the boundary curves $m_1$ and $m_2$. Let $\phi(t,r)$ be a smooth function on $N$, $0\leq \phi \leq 1$, $\phi(t,0)=1$, and let $\eta(t)$ be the function $R_{m}(t)-R$ on $\partial H$. It follows that $R+\phi(t,r)\eta(t)$ is a Morse function when restricted to $\partial H$, i.e. for $r=0$. \\

To construct a function without degenerate critical points on a neighbourhood of the boundary, it is enough to show that $\phi(t,r)$ can be chosen such that $\frac{d(R+\phi(t,r)\eta(t))}{dr}$ and $\frac{d(R+\phi(t,r)\eta(t))}{dt}$ are not simultaneously zero on a neighbourhood $N_1$ of $\partial H$ contained in $N$.\\

As a consequence of smoothness, $\frac{d(R+\phi(t,\kappa)\eta(t))}{dt}-\frac{d(R+\phi(t,0)\eta(t))}{dt}$ can be made arbitrarily small by choosing $\kappa$ sufficiently small. Since $R+\phi(t,r)\eta(t)$ is a Morse function on $\partial H$, when restricted to $\partial H$, $\frac{d(R+\phi(t,r)\eta(t))}{dt}$ is only zero at (isolated) critical points $p_{1}=(t_{1},0)$, $p_{2}=(t_{2},0),\ldots ,p_{n}=(t_{n},0)$. Therefore, $N_{1}\subset N$ can be chosen such that in $N_1$, $\frac{d(R+\phi(r,t)\eta(t))}{dt}$ can only pass through zero in a neighbourhood of the form $P_{i}:=(p_{i}-\epsilon, p_{i}+\epsilon)\times (0,\epsilon)$, for $i=1,2,\ldots ,n$. Inside each of the $P_i$, $\phi$ can be chosen such that $\frac{d(R+\phi(r,t)\eta(t))}{dr}$ is nonzero. This is possible because $\epsilon$ can be chosen such that $R$, $\eta$ and their derivatives do not vary much in the $\epsilon$ neighbourhoods. It follows that $N_1$ and $\phi$ can be chosen such that $R+\phi(r,t)\eta(t)$ is a Morse function on $N_1$. \\

\textbf{Preserving embeddedness} It remains to show that when the homotopy taking $H$ to $H^{'}$ is chosen to be sufficiently close to the identity, embeddedness is preserved. Let $H^1$ be a (possibly immersed) surface with boundary in \Sr\ that coincides with $H$ outside of $N$ and is given by the graph $(s, R+\phi(r,t)\eta(t))$ in the coordinate chart $(U_{i}, s_{i}, R)$ over $N$. Since $H$ is smoothly embedded in \Sr\  as a submanifold with boundary, it follows from theorems 6.1 and 6.3 of \cite{Hirsch} that $H$ has an embedded neighbourhood $\mathcal{E}(H)$ in \Sr. As $R_m$ approaches $R$ in the $\mathcal{C}^2$ topology on $\partial H$, $R+\phi(r,t)\eta(t)$ also approaches $R$ in the $\mathcal{C}^2$ topology on $N$. If $R_m$ was chosen to be sufficiently close to $R$ in the $\mathcal{C}^2$ topology, it follows that $H^1$ is contained in $\mathcal{E}(H)$ and is also embedded.

Setting $H^{'}=H^1$ for $R_m$ sufficiently close to $R$ therefore gives a surface with the properties claimed in the statement of the lemma.
\end{proof}

\section{Ordinary handles and bow tie Handles}

In the proof of theorem \ref{minimalgenus} it is necessary to keep track of intersection properties of projections of curves and arcs to $S\times 0$. For this reason it is helpful to distinguish between two distinct methods of attaching handles, depending on the way the handle projects into $S\times 0$. \\

Let $H^{a}_{b}:= H\cap (S\times (b,a) )$, $H_{b}:=H\cap (S\times  (b, \infty))$, $H^{a}:=H\cap((-\infty, a))$ and $H(a):=H\cap (S\times a)$.\\

Suppose the closure of $H^{a}$ contains two components, $F_1$ and $F_2$, that are subsurfaces of a connected component $F$ of $H^{a+\delta}$. In two dimensions, a handle can be thought of as an oriented rectangle $Q$ in $S\times \mathbb{R}$, whose boundary is a union of four arcs, each given an orientation as a subarc of the boundary of $Q$. To attach a handle to $F_{1}\cup F_{2}$, a pair of opposite sides of $Q$, $q^1$ and $q^2$, are glued along arcs on the boundary components of $F_1$ and $F_2$ respectively, in such a way that pairs of arcs with opposite orientation are glued together. In this way, an oriented surface $F$ is obtained, such that $F_1$ and $F_2$ are oriented as subsurfaces of $F$. \\
\subsection{Bow tie handles}
As a result of working with surfaces whose boundary is not contained in two level sets, it is necessary to consider an additional possibility; informally, an ordinary handle with a half twist in it. Note that a ``handle'' of this type does not come about from a critical point in the \textit{interior} of the surface. It is not hard to check that cutting open an ordinary handle and regluing it with a half twist can reduce the number of critical points in the interior of the surface. As illustrated in figure \ref{celldecomp}, a necessary condition for the existence of this phenomenon is the existence of critical points on the boundary of the surface.\\

Consider an example in which $F_1$ and $F_2$ project one to one onto subsurfaces of $S\times 0$ with opposite orientations. Then the handle $Q$ has to be embedded in $S\times \mathbb{R}$ with an odd number of half twists, otherwise the orientations of $F_1$ and $F_2$ can not be made to match up.\\






\textbf{Definition of bow tie handle.} To sum up: suppose that for a given fixed $a$, $H^{a+\delta}$ is obtained from $H^{a}$ by attaching opposite sides of an oriented rectangle $Q$ to arcs on $\partial H^{a}$. This is done in such a way that the orientation on $Q$ matches that on $H^{a}$. Suppose also that the projection of $\partial H^{a}$ to $S\times 0$ only has essential intersections.
Let $h(t)$ be a homotopy of $H^{a+\delta}$ in $S\times \mathbb{R}$ such that
\begin{itemize}
\item $h(t)$ fixes $H^{a}$
\item There are only essential points of intersection between the projection to $S\times 0$ of the image of $Q$ under $h(1)$, and $\partial H^{a}$. 
\end{itemize}
If there does not exist a $h(t)$ such that $h(1)Q$ projects one to one onto its image in $S\times 0$, then it will be said that $H^{a+\delta}$ is obtained from $H^{a}$ by attaching a bow tie handle.



\begin{figure}
\includegraphics[width=9cm]{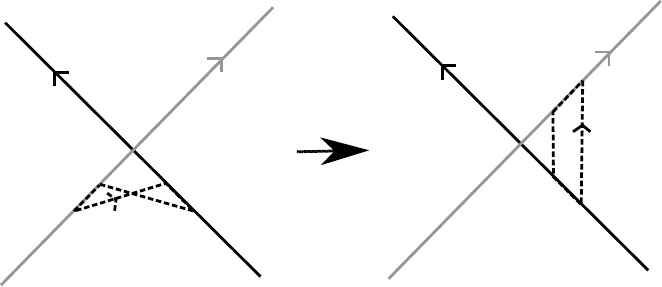}
\caption{A fake bow tie handle is shown on the left. A surface with this handle attached to the boundary is homotopic to the surface with an ordinary handle attached, as shown on the right. The long diagonal line segments are the one skeleton of the cell decomposition of the handle, as shown in figure \ref{celldecomp}}
\label{doublecrossingscumbag}
\end{figure}

If $Q$ is a bow tie handle that is attached to the surface $F$ and $F\cup Q$ is homotopic to the surface $F$ with an ordinary handle attached, $Q$ will be called a \textit{fake bow tie handle}. An example is given in figure \ref{doublecrossingscumbag}.\\


\begin{figure}
\begin{center}
\def\svgwidth{13cm}
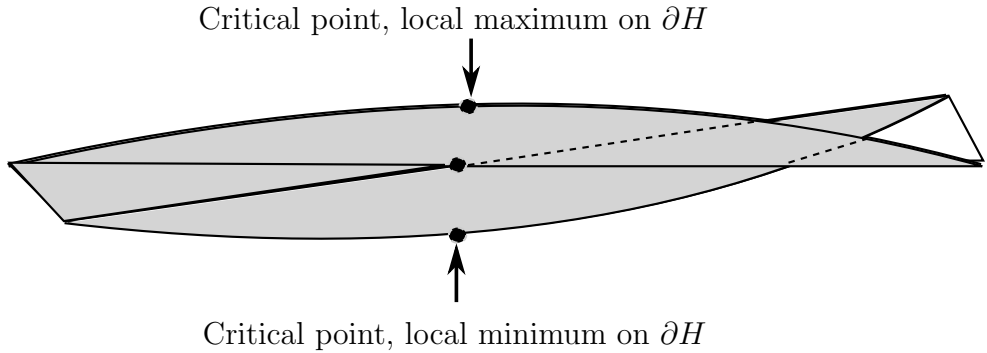
\caption{Cell decomposition of a bow tie handle. One side of the handle is shown in grey.}
\label{celldecomp}
\end{center}
\end{figure}


\subsection{The types of handles and classification of level sets}
Since $H$ is embedded, whenever $a$ is not a critical value, $H(a):=H\cap (S\times a)$ is a union of curves and arcs that project one to one into $S\times 0$. Otherwise, choose the representative of the homotopy class of $H$ such that a Morse function $m_{R}$ is obtained and there is at most one critical point for any value of $m_R$. Then $\delta$ can be made small enough to ensure that $H^{a+\delta}$ is obtained from $H^{a-\delta}$ either by gluing a disc along a boundary curve, taking a disjoint union with a disc or by attaching a single (ordinary) handle. \\

\textbf{Gluing a 2-disc along a boundary component.} Whenever the boundary of the 2-disc is glued along a contractible curve of the boundary of $H^{a-\delta}$, $H(a+\delta):=H\cap (S\times \left\{a+\delta\right\})$ is a union of curves and arcs that project one to one into $S\times 0$. Similarly, whenever a 2-disc is attached to the boundary of $H_a$ along a boundary arc, as in figure \ref{celldecomp}, $H(a):=H\cap (S\times a)$ is also a union of curves and arcs that project one to one into $S\times 0$. The diagram in figure \ref{celldecomp} is a bit misleading here, since the existence of horizontal cells preclude the possibility of the hight function being a Morse function.\\


\textbf{Attaching an (ordinary) handle.}
If $a$ is a critical value and $H^{a+\delta}$ is obtained from $H^{a}$ by attaching a handle, $H(a)$ is a one dimensional cell complex. There is a point $p$ on $\partial H^{a}$ such that $p$ has a neighbourhood in the closure of $H^{a}$ consisting of two 2-cells that meet at the vertex $p$. By definition, $p$ has a neighbourhood in $H$ that projects one to one onto a neighbourhood of $\pi(p)$, so the point $p$ in $H(a)$ is a point at which two arcs in $H(a)$ touch but do not cross over.\\

\subsection{Proof of theorem in the absence of bow tie handles}
\begin{proof}[Proof of Theorem \ref{minimalgenus}]

Let $c$ be a simple curve in the intersection of $H$ with a level set of $\mathbb{R}$, $S\times \left\{a\right\}$. An \textit{up collar} of $c$ is an annular subsurface of $H$, $c\times \left[0,\epsilon\right]$, such that $c\times \left(0,\epsilon\right]\subset S\times \left(a,\infty\right)$ and $c\times \left[0,\epsilon\right]$ is not contained in a component of $H\cap S\times \left[a, \infty\right)$ consisting of an annulus with core curve $c$ or a punctured sphere with boundary curves either contractible or homotopic to $c$.
A \textit{down collar} is defined analogously.\\



Let $c_0$ be a curve in the multicurve $m_1$ on $\partial H$. Since $c_0$ is a simple curve, it is possible to assume without loss of generality that the zero of $R$ was chosen to contain the boundary curve $c_0$ of $H$.

\begin{lemma}[Existence of an up/down collar]
Assume $H$ has no bow tie handles. There exists an embedded representative of the homotopy class of $H$ such that, by mapping $R$ to $-R$ if necessary, the boundary curve $c_0$ has an up collar.  
\label{collar}
\end{lemma}

\begin{proof}
$H$ is orientable, so $c_0$ could not be the core curve of a M\"obius band. Therefore $c_0$ is a boundary curve of an annulus $A$ contained in $H$.\\

By cutting an annulus off one boundary component of $H$ if necessary, it is possible to assume without loss of generality that $c_0$ is not on the boundary of a component of $H\cap (S\times \left[a, \infty\right))$ consisting of an annulus with core curve $c_0$ or a punctured sphere with boundary curves either contractible or homotopic to $c_0$. (Since $H$ is assumed to be incompressible, a curve on $H$ that is contractible in $S\times \mathbb{R}$ is also contractible in $H$.)\\

Note that since the zero of the $R$ coordinate is being implicitly assumed to contain the boundary curve $c_0$, cutting an annulus with core curve $c_0$ off the boundary could give rise to a new Morse function, $m_{R}^{'}$. The handle decomposition of the surface might contain bow tie handles with respect to a new Morse function. However, in this case, the new boundary curves are in the same level set as the previous boundary curves, so this problem does not need to be considered here.\\

All that could go wrong is therefore that the second boundary curve of $A$ might have nonzero winding number in $S\times \mathbb{R}$ around $c_0$. Let $\dot{c}_{0}(t)$ be a nonvanishing tangent vector to the curve $c_0$, let $\dot{r}(t)$ be a nonvanishing vector field along $c_0$ tangent to $A$ and linearly independent to $\dot{c}_{0}(t)$, and let $n$ be a normal vector to $S\times \left\{a\right\}$. If the second boundary curve of $A$ has nonzero winding number around $c_0$, the handedness of $(\dot{c}_{0}(t), \dot{r}(t), n)$ has to change when moving around $c_0$. Therefore, $A$ has to be constructed by gluing together cells, some of which project to cells in $S\times 0$ with the induced subsurface orientation, and some of which project to cells in $S\times 0$ with the opposite of the induced subsurface orientation i.e. $A$ has to contain bow tie handles.

\end{proof}

\begin{remark}
What the previous lemma does not show is that there is an embedded representative of the homotopy class containing $H$ whose intersection with $S\times 0$ contains the multicurve $m_1$ such that every curve in $m_1$ simultaneously has an up collar. 
\label{remark}
\end{remark}

\subsection{Choosing the zero of $\mathbb{R}$} It is very convenient to choose $S\times 0$ such that the boundary curve $c_0$ is contained in $S\times 0$, however this choice results in a boundary curve $c_0$ consisting of degenerate critical points. This detail is resolved in the next lemma.

\begin{lemma}
Suppose the zero of $R$ was chosen such that the boundary curve $c_0$ of $H$ is contained in $S\times 0$. Then there exists an embedded surface $H_{\text{wiggled}}$ homotopic to $H$ with the following properties:
\begin{enumerate}
\item $H_{\text{wiggled}}$ is arbitrarily close to $H$ in the Hausdorff topology
\item the restriction of $R$ to $H_{\text{wiggled}}$ is a Morse function, and
\item there exists a noncritical value $r$ of $R$ such that $S\times \left\{r\right\}$ intersects $H_{\text{wiggled}}$ along the collar of $c_0$ whose existence was shown in the previous lemma such that $H_{\text{wiggled}}(r)$ contains a curve homotopic to $c_0$.
\end{enumerate}
\label{boring}
\end{lemma}

\begin{proof}

It follows from lemma \ref{morsetheory} that there is a surface $H_{\text{wiggled}}$ homotopic to $H$ and arbitrarily close to $H$ in the Hausdorff topology to which the restriction of the $R$ coordinate is a Morse function. Whenever $H_{\text{wiggled}}$ is sufficiently close to $H$ in the Hausdorff topology, the boundary curve $c_0$ has to have a collar in $H_{\text{wiggled}}$ such that the intersection of $S\times \left\{r\right\}$ with this collar contains a curve homotopic in $H_{\text{wiggled}}$ to the boundary curve $c_0$, for some small, noncritical value $r$ of $R$.  
\end{proof}

It can therefore be assumed without loss of generality that the zero of the $R$ coordinate and the embedded representative of the homotopy class of $H$ are chosen such that the restriction of $R$ to $H$ is a Morse function, $H(0)$ contains a curve homotopic to the boundary curve $c_0$ and no two critical points occur at the same value of $R$. With this choice of the zero of the $R$ coordinate, the boundary curves $m_2$ and $m_{1}\setminus c_{0}$ might intersect $S\times 0$ in a complicated way.

\subsection{Reordering Critical Points}\label{reorder}
There are standard results in Morse theory that describe how to modify a Morse function so as to change the order in which handles are attached. For this proof, it is also necessary to restrict to Morse functions that can be realised as the projection to the $\mathbb{R}$ coordinate of an embedded surface. Let $c_{0}^{'}$ be a curve on $H$ homotopic to the boundary curve $c_0$. Since $c_{0}^{'}$ has self intersection number 0, it is possible to find a new metric on $S\times \mathbb{R}$ so that $c_{0}^{'}$ is in the 0 level set of the $\mathbb{R}$ coordinate. This new metric will induce a new Morse function $m_{R^{'}}$. \\

Let $p$ be a critical point of $m_{R}$ corresponding to, for example, an ordinary handle. If $c_{0}^{'}$ does not pass through $p$, let $N$ be a small neighbourhood of $p$ disjoint from $c_{0}^{'}$, and assume that the new metric on $S\times \mathbb{R}$ was obtained without altering the metric inside $N$. Then if $c_{0}^{'}$ passes sufficiently close to $p$, the Morse function $m_{R^{'}}$ has the property that the critical point $p$ gives rise to the first surgery performed on the annulus with core curve homotopic to $c_0$. In this way, it is possible to assume without loss of generality that the first surgery performed on the annulus with core curve homotopic to $c_0$ involves attaching an ordinary handle.\\

When modifying the Morse function as described in the previous paragraph, it might happen that the assumption of no bow tie handles breaks down. We do not worry about this here, and show how to deal with bow tie handles in a later section.\\

\section{Handle decomposition}\label{Handledecomp} It is finally possible to start the handle decomposition of $H$. This handle decomposition will be constructed such that each of the pants is homotopic to a pant in $S\times 0$. A surface homotopic to $H$ is obtained by gluing the pants together, in a manner similar to the construction in section \ref{surfaceconstruction}. This pant decomposition is then used to obtain a convenient Morse function, from which a path in the curve complex $\mathcal{HC}(S, [m_{1}])$ is obtained. \\

If $a$ is so small that there are no critical points of $R$ in the interval $\left[0,a\right]$, then $H^{a}_0$ is a union of annuli whose core curves project onto a multicurve in $S$ and perhaps some simply connected components. \\

Suppose now that $a$ is large enough to ensure that there is only one critical value, $b$, in the interval $\left[0,a\right]$. If $H^{a}_0$ contains a simply connected component that intersects some $S\times (a-\delta)$ along an arc or a contractible curve, and if this component was not in $H^{x}_0$ for $x<b$, then the critical point has not changed the topology of the component of $H^{a}_0$ with $c_0$ on its boundary. \\

By the argument in subsection \ref{reorder}, it can be assumed without loss of generality that there is not a local extremum of the Morse function at height $b$ and that $b$ is not a critical level that gives rise to a surgery that connects a contractible subsurface to the annulus with core curve homotopic to $c_0$.\\







The critical point at height $b$ is therefore a saddle point. In particular, $H^{a}_0$ is obtained from $H^{b}_0$ (a disjoint union of contractible components and annuli whose core curves project onto a multicurve in $S$) by attaching a handle. By construction, it follows that the endpoints of the handle are either both on the same annulus or on two different annuli. Whenever both of the endpoints of the handle are on the boundary of the annulus with core curve $c_0$, $H^{a}_0$ contains a pair of pants with boundary curves $c_0$ and $\eta \cup \beta$. \\

\textbf{Remark.} There is another alternative here, namely that the handle has one endpoint on each boundary component of the annulus with core curve $c_0$. However, this does not happen here, because the handle is attached to the boundary component $H(b)$ which only contains one curve homotopic to $c_0$.\\

If neither $\eta$ nor $\beta$ is contractible, $\eta \cup \beta$ is homotopic to a multicurve because it is a submanifold of the intersection of the embedded surface $H$ with $S\times \left\{b\right\}$.\\

If none of the curves are contractible, $c_{0}\cup \eta \cup \beta$ is also a multicurve, because $\eta \cup \beta$ is constructed by attaching a single handle to $c_0$, where the handle is a subsurface of $S\times \left\{a\right\}$ without self intersections that meets the projection of $c_0$ onto $S\times \left\{a\right\}$ only at its endpoints. It follows that the corresponding pair of pants is homotopic to a pair of pants in $S\times 0$. \\

Let $\delta_{1}$ be the multicurve $\eta \cup \beta$, unless one of $\eta$ or $\beta$ is contractible. If one of $\eta$ or $\beta$ is contractible, $\beta$ for example, it follows from the assumption of incompressibility of $H$ that $\beta$ bounds a disc $B$ in $H$. If the disc $B^{'}$ in $S\times \left\{a\right\}$ bounded by $\beta$ is disjoint from $H$, since $S\times \mathbb{R}$ is aspherical, there exists a homotopy of $H$ that fixes $H$ outside of the closure of $B$ and takes $B$ to $B^{'}$. If $B^{'}$ intersects $H$, it intersects $H$ in a union of disjoint circles, and we perform a series of homotopies to remove the intersections, starting with an innermost circle. Although these homotopies do not result in a surface on which the $\mathbb{R}$ coordinate defines a Morse function, in practice this is not a problem due to the fact that Morse functions are dense.\\



Similarly if one of the endpoints of the handle is on the boundary of the annulus with core curve $c_0$ and the other is on the boundary of another annulus whose core curve $\eta$ is in the multicurve $m_1$. In this case $\beta$ is the curve obtained by connecting the annuli with core curves $c_0$ and $\eta$ by a handle, and $c_{0}\cup \eta \cup \beta$ is a multicurve for the same reason as in the previous case. In this case, let $\delta_1$ be the multicurve $\eta \cup \beta$ unless one of $\eta$ or $\beta$ is contractible. \\

If the handle doesn't have an endpoint on the annulus with core curve $c_0$, then the intersection of $H_0$ with $S\times \left\{a\right\}$ will be a union of arcs plus a new multicurve, $m_{1*}$, containing $c_0$. That $m_{1*}$ is a multicurve follows from the same argument as before.\\

Following lemma \ref{collar} it can be assumed without loss of generality that the component of $H_0$ with the boundary curve $c_0$ of $H$ does not consist of an annulus with core curve $c_0$ or a punctured sphere whose boundary curves are either contractible or homotopic to $c_0$. Therefore, if $a$ is increased enough, there will be a critical point of $m_R$ on the component of $H^{a}_0$ with $c_0$ on its boundary. Since there are only finitely many critical points, eventually the desired pair of pants is obtained, and $\delta_1$ is then defined to be $\eta \cup \beta$.\\

To construct $\delta_2$, cut the pair of pants with boundary $c_{0}\cup \delta_1$ off $H$ to obtain an embedded surface $H_1$ whose boundary contains the curves $\delta_1$. Since $\delta_1$ is a multicurve, the previous argument can be applied with a curve from $\delta_1$ in place of $c_0$ and $H_1$ in place of $H$. In this way, a union of multicurves, $\delta_{1}, \delta_{2}, \ldots, \delta_n$ that decompose $H$ into a union of subsurfaces is obtained, where each of these subsurfaces are homotopic to subsurfaces of $S\times 0$. \\

\begin{figure}
\begin{center}
\def\svgwidth{13cm}
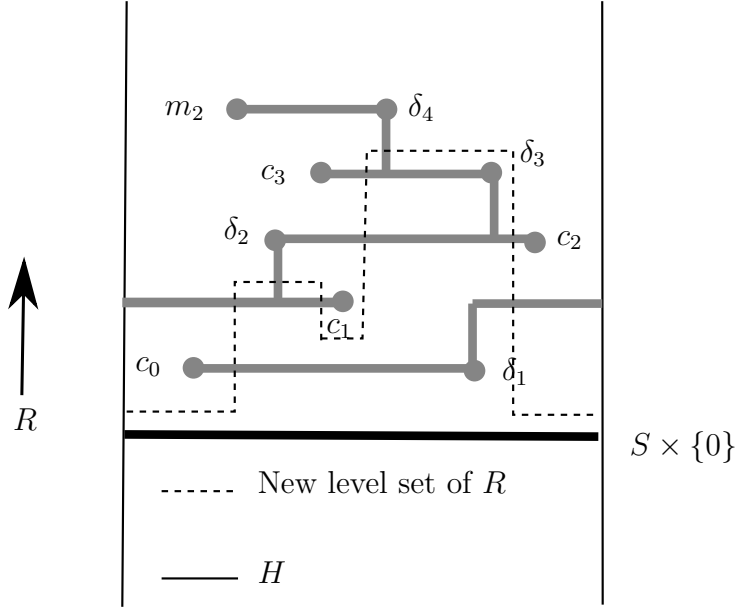
\caption{A choice of the zero of $R$ satisying remark \ref{remark} is represented by the dotted line. Let $m_{1}=c_{0}\cup c_{1}\cup c_{3}\ldots$. Horizontal lines represent subsurfaces of level sets, vertical lines represent annuli or unions of annuli. A dot represents a curve or multicurve.}
\label{remarkdiagram}
\end{center}
\end{figure}

The proof is not finished yet, because $m_{1}, \delta_{1}, \delta_{2},\ldots\delta_{n}, m_2$ is not in general a path in $\mathcal{HC}(S, \left[m_{1}\right])$. Although for example, $\delta_1$ does not intersect $c_0$, it might intersect other curves in $m_1$. It is enough to show that all the curves in $m_1$ simultaneously have an up collar, because then the pant decomposition constructed as described in the previous paragraphs defines a path $m_{1}, \gamma_{1}, \ldots, \gamma_{j}, m_2$ in $HC(S, \left[m_{1}\right])$, where $\gamma_{1}:=(m_{1}\cup\eta\cup\beta)\setminus c_0$ if the handle has both endpoints on the annulus with core curve $c_0$ or $\gamma_{1}:=(m_{1}\cup\beta)\setminus (c_{0}\cup \eta)$ otherwise.\\

The pant decomposition given by the $\left\{\delta_{i}\right\}$ shows that $H$ is homotopic to a surface obtained by gluing together pants along boundary curves, where each of these pants is homotopic to an incompressible subsurface of $S\times 0$, similar to the construction in section 2. This is used to show that the zero of $R$ can be chosen such that all curves in $m_1$ simultaneously have an up collar. How to construct such a zero for the $R$ coordinate is explained schematically in figure \ref{remarkdiagram}. This completes the proof of theorem \ref{minimalgenus} in the absence of bow tie handles.\\



\textbf{Remark - disconnected and immersed surfaces.} Embeddedness of the surface was used in the proof of theorem 1.1 to ensure that intersections with level sets are embedded. For disconnected surfaces, the notion of embeddedness is not strong enough, because the connected components of the surface might intersect when boundary curves are forced to lie in the same level sets. A weaker statement is obtained in lemma \ref{weak} for immersed surfaces.

\subsection{Proof with bow tie handles} It remains to prove the theorem in the case of bow tie handles.\\

Examples of orientable surfaces in $S\times \mathbb{R}$ with boundary $m_{2}-m_{1}$, whose handle decomposition contains bow tie handles are not difficult to construct. For example, given a simple path $m_{1}, \gamma_{1}, m_{2}$, construct a surface in $S\times \mathbb{R}$ by gluing together two subsurfaces that project onto subsurfaces of $S\times 0$; one with the subsurface orientation of $S\times 0$ and one with the opposite orientation. \\

The key observation here is that the resulting bow tie handles occur in pairs, otherwise the boundary of the surface could not be a union of two multicurves. Similarly, when constructing paths in $\mathcal{HC}(S, \alpha)$, points of intersection were removed in pairs, where each pair consisted of intersections with opposite handedness. When attempting to construct a pants decomposition of $H$ corresponding to a path $m_{1}, \gamma_{1}, \gamma_{2},\ldots ,m_{2}$ in $\mathcal{HC}(S,\alpha)$, no $\gamma_i$ is not allowed to separate a bow tie handle in $H$ from all possible partners. The argument in this section shows that this restriction is sufficient to obtain a path $m_{1}, \gamma_{1}, \gamma_{2},\ldots ,m_{2}$ in $\mathcal{HC}(S,\alpha)$ from which $H$ is constructed. \\




\begin{figure}
\centering
\includegraphics[width=13cm]{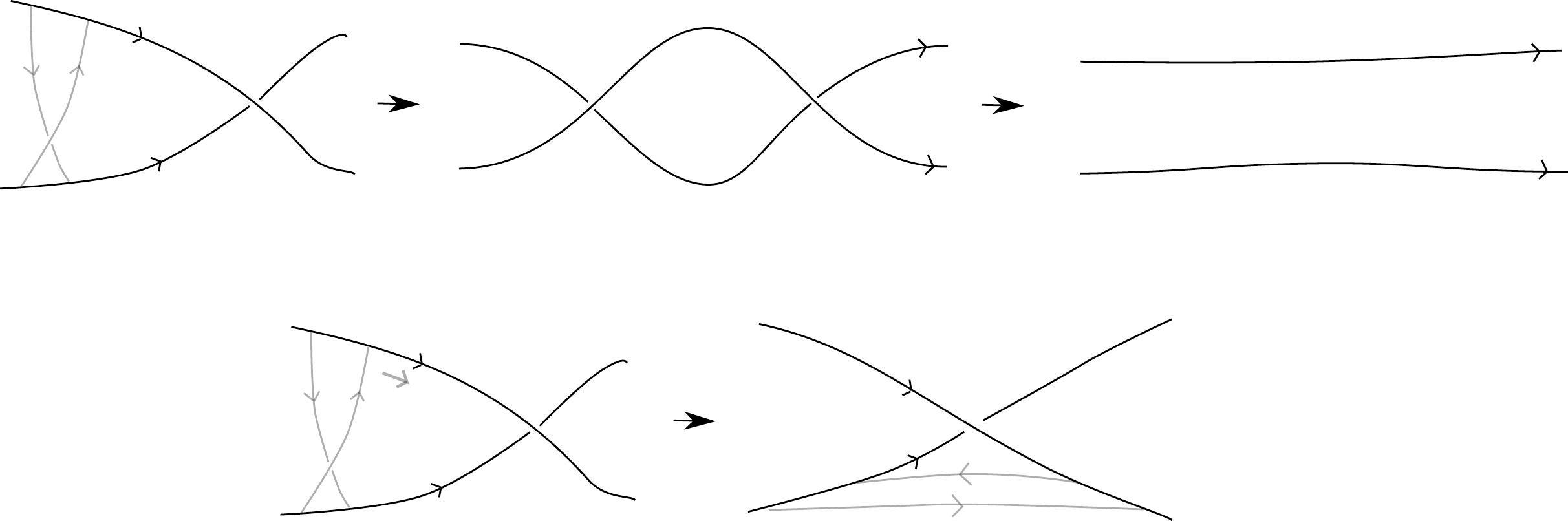}
\caption{A bow tie handle that decreases intersection number is fake.}
\label{fakingit}
\end{figure}


\textbf{A nontrivial handle decomposition of the subsurface of $H$ with boundary $m_{2}-\gamma_{i}$ can not consist of bow tie handles only, unless they are fake.} In \cite{FarbMargalit}, it was proven in Proposition 1.7 that two simple transverse curves  $a$ and $b$ on $S$ are not in minimal position iff there is a bigon in $S$ with boundary consisting of a subarc of $a$ and a subarc of $b$. The arguments given in the proof also show that if $a$ is not simple, $a$ is not in minimal position iff there is a subarc of $a$ that bounds a disk in $S$ or there is a bigon in $S$ bounded by two subarcs of $a$. It follows that if a bow tie handle is attached to two simple, disjoint, boundary curves, the resulting boundary curve will be simple iff at least one of the curves is contractible. Similarly, if attaching a bow tie handle gives new boundary curves with smaller (self)intersection number, the bow tie handle is necessarily fake. This is illustrated in figure \ref{fakingit}, where it is shown that the bow tie handle can be isotoped through the bigon to obtain an ordinary handle. Since attaching bow tie handles creates points of intersection on the boundary curves, and can not remove points of intersection amongst the boundary curves, the claim follows.\\

Remark - The claim in the previous paragraph is not true for pant decompositions of surfaces in $S\times \mathbb{R}$ in general. For example, consider an embedded pant $P$ in $S\times \mathbb{R}$ with boundary curves $a, b$ and $c$, where $a$ projects onto a simple curve in $S\times \left\{0\right\}$, and $b$ and $c$ do not. Since the curves $b$ and $c$ can not be contained in level sets of $\mathbb{R}$, the Morse function $m_{R}$ on $P$ will have critical points coming from local extrema on the boundary curves $b$ and $c$, and the corresponding surgeries can not in general be paired up to give a decomposition by ordinary handles.\\



Let $i$ be as large as possible such that the deltas can be constructed as in the previous section without encountering bow tie handles, and let $\delta_{i}^{'}$ be a multicurve homotopic to $\delta_i$, where $\delta_{i}^{'}$ is on the boundary of the subsurface of $H$ obtained by cutting off the first $i$ pants. Since $\delta_{i}^{'}$ is a multicurve, as in subsection \ref{reorder}, it is possible to find a new metric on $S\times \mathbb{R}$ so that $\delta_{i}^{'}$ is in the zero level set of the $\mathbb{R}$ coordinate. With respect to this new zero of the $\mathbb{R}$ coordinate, let $p$ be a critical point whose corresponding surgery consists of attaching an ordinary handle. Then if $\delta_{i}^{'}$ passes sufficiently close to $p$, the Morse function $m_{R^{'}}$ has the desired properties. In particular, with respect to $m_{R^{'}}$, the ``next critical point'' above $\delta_{i}^{'}$ gives rise to a surgery that consists of attaching an ordinary handle. In this way, the multicurve $\delta_{i+1}$ is constructed. This argument can be repeated with the subsurface of $H$ obtained by cutting off the first $i+1$ pants.\\

\end{proof}

\subsection{Immersed Surfaces}
One reason for studying $\mathcal{HC}(S,\alpha)$ is that it can be used for constructing minimal genus surfaces in $S\times \mathbb{R}$. Not all such minimal genus surfaces are embedded, as shown in \cite{Me2} example 13. For immersed surfaces in $S\times \mathbb{R}$, a slightly weaker form of theorem \ref{minimalgenus} is obtained.

\begin{lemma}
\label{weak}
Suppose $H$ is an oriented, immersed, incompressible surface in $S\times \mathbb{R}$. Then $H$ is homotopic to a union of subsurfaces $T_i$ glued along homotopic boundary curves, as described in section \ref{surfaceconstruction}. Each of the $T_i$ is homotopic to an embedded subsurface of $S\times 0$.

\end{lemma}
\begin{proof}
If $H$ is not embedded, it follows from standard arguments based on theorem 3.3 of \cite{Scott} that there is a finite index covering space, $\tilde{S}\times \mathbb{R}$, of $S\times \mathbb{R}$, such that $H$ can be lifted to an embedded surface, $\tilde{H}$, in $\tilde{S}\times \mathbb{R}$. Let $\tilde{\alpha}$ be the integral homology class of $\tilde{S}$ containing the lift of $m_1$ to the boundary of $\tilde{H}$.\\

The same arguments as in the proof of theorem \ref{minimalgenus} show that $\tilde{H}$ is homotopic to a surface constructed from a path $\tilde{m}_{1}, \tilde{\gamma}_{1}, \tilde{\gamma}_{2},\ldots,\tilde{m}_2$ as described in section \ref{surfaceconstruction}. The surface $\tilde{S}$ can be covered by a finite number of neighbourhoods that project one to one onto $S$, and so therefore can $\tilde{H}$, from which the lemma follows. 
\end{proof}

Note that it is not possible to argue, as in the proof of theorem \ref{minimalgenus}, that the decomposition of $H$ into the subsurfaces $T_i$ determines a path in $\mathcal{HC}(S,\alpha)$. All that can be shown is that each curve in $\tilde{\gamma}_i$ projects onto a curve in $S$, and that the projection of $\partial T_i$ is a multicurve. It does not follow that the projection of $\tilde{\gamma}_i$ is a multicurve. \\





\bibliographystyle{plain}

\bibliography{bib2}

\end{document}